\documentclass[11pt,leqno]{amsart}
\usepackage{amscd}
\usepackage{color}
\usepackage[symbol]{footmisc}
\usepackage{amssymb}
\usepackage{amsfonts}
\usepackage{latexsym}
\usepackage{verbatim}
\usepackage{bbm}

\newcommand{\R}{\mathbb{R}}

\renewcommand{\H}{\mathbb{H}}

\renewcommand{\epsilon}{\varepsilon}
\renewcommand{\theta}{\vartheta}
\renewcommand{\phi}{\varphi}
\renewcommand{\Re}{\mathrm{Re}}

\newcommand\abs[1]{\left\lvert#1\right\rvert}

\theoremstyle{plain}
\newtheorem{teor}{Theorem}

\newtheorem{lem}[teor]{Lemma}

\theoremstyle{plain}
\newtheorem*{rem}{Remark}

\theoremstyle{definition}

\title{A parabolic approach to the Calabi-Yau problem in HKT geometry}

\begin{document}
 
\thanks{This work was supported by GNSAGA of INdAM}

\address{(Lucio Bedulli) Dipartimento di Ingegneria e Scienze dell'Informazione e Matematica, Universit\`a dell'Aquila,
via Vetoio, 67100 L'Aquila, Italy}
\email{lucio.bedulli@univaq.it}

\address{(Giovanni Gentili, Luigi Vezzoni) Dipartimento di Matematica G. Peano \\ Universit\`a di Torino\\
Via Carlo Alberto 10\\
10123 Torino\\ Italy}
\email{giovanni.gentili@unito.it\,,\quad luigi.vezzoni@unito.it}

\author{Lucio Bedulli, Giovanni Gentili and Luigi Vezzoni}

\date{\today}

\maketitle
\begin{abstract}
We consider the natural generalization of the parabolic Monge-Amp\`ere equation to HKT geometry.  
We prove that in the compact case the equation has always a short-time solution and when the hypercomplex structure is locally flat and admits a compatible hyperk\"ahler metric, then the equation has a long-time solution whose normalization converges to a solution of the quaternionic Monge-Amp\`ere equation first introduced in \cite{Alesker-Verbitsky (2010)}. The result gives an alternative proof of a theorem of Alesker in \cite{Alesker (2013)}.
\end{abstract}

\section{Introduction}
Let $M$ be a $4n$-dimensional smooth manifold. A {\em hypercomplex} structure on $M$ is a triple $(I,J,K)$ of complex structures satisfying the quaternionic identities
$$
IJ=-JI=K\,.
$$  
A Riemannian metric $g$ on $(M,I,J,K)$ is called {\em hyperhermitian} if it 
is compatible with each of $I,J,K$.  Every hyperhermitian metric induces the form
$$
\Omega=\omega_J+i\omega_K\in \Lambda_I^{2,0}\,,
$$
where $\omega_J=g(J\cdot,\cdot)$ and $\omega_K=g(K\cdot,\cdot)$ are the fundamental forms of $(g,J)$ and $(g,K)$ respectively. The form $\Omega$ is nondegenerate, i.e. $\Omega^n \neq 0$ everywhere, and determines $g$ by the relation
\begin{equation}\label{Omega_g}
\Omega(X,Y)=2g(JX,Y)
\end{equation}
for every $X,Y$ vector fields of type $(1,0)$ with respect to $I$.   
 
A hyperhermitian manifold $(M,I,J,K,g)$ is called HKT (hyperk\"ahler with torsion) if 
$$
\partial \Omega=0\,,
$$
where $\partial$ is with respect to $I$. HKT manifolds were first introduced by Howe and Papadopoulos in 	
\cite{Howe-Papadopoulos (2000)}, but the characterization in terms of the form $\Omega$ is due to Grantcharov and Poon \cite{Grantcharov-Poon (2000)} who also showed the existence of HKT structures on some homogeneous hypercomplex manifolds.  HKT structures have been studied intensively in the last years and many analogies to K\"ahler manifolds  have been discovered  (see e.g. \cite{Alesker (2013),Alesker-Shelukhin (2017),Alesker-Verbitsky (2006),Alesker-Verbitsky (2010),Banos,DinewSroka,GF,GentiliVezzoni,GLV,Grantcharov-Poon (2000),Ivanov,Sroka,Swann,Verbitsky (2002),Verbitsky (2007),Verbitsky (2009)} and the references therein). In particular a pluripotential theory has been developed in HKT geometry, according to which every HKT form $\Omega$ can be locally written as 
$$
\Omega=\partial \partial_J v
$$
for some real smooth function $v$  (see \cite{Alesker-Verbitsky (2006)} and \cite{Banos}), where $\partial_J=J^{-1}\bar \partial J$ and the action of $ J $ on a $ k $-form $ \alpha $ is defined as
\[
J\alpha(X_1,\dots,X_k)=\alpha(JX_1,\dots,JX_k)\,.
\]
Furthermore, in analogy with the complex case, the space of {\em quaternionic plurisubharmonic functions} has been introduced 
$$
\mathcal{H}_\Omega=\{\varphi \in C^{\infty}(M)\,\,:\,\, \Omega_\varphi:=\Omega+\partial	\partial_J	\varphi>0\}\,,
$$
where \lq\lq $\Omega_\varphi > 0$\rq\rq\, means that $\Omega_\varphi$ induces a hyperhermitian metric via \eqref{Omega_g}. 

In \cite{Alesker-Verbitsky (2010)} Alesker and Verbitsky introduced the following quaternionic version of the Calabi problem in analogy with the famous complex Calabi problem solved by Yau in \cite{yau}. Let $(M, I, J, K, g)$ be a compact HKT manifold and let $F\in C^{\infty}(M)$ be a smooth real valued function. The {\em quaternionic Monge-Amp\`ere equation} is
\begin{equation}\label{QMAe}
(\Omega+\partial \partial_J\varphi)^n=b\,{\rm e}^F\,\Omega^n\,,
\end{equation}
where $F\in C^{\infty}(M)$ is the datum, while $(\varphi,b)	\in	\mathcal{H}_\Omega \times 	\R_+  $  is the unknown.
Equation \eqref{QMAe} is most naturally settled if the canonical bundle of $(M,I)$ is holomorphically trivial. In this case there exists a holomorphic volume form $\Theta$ on $(M,I)$ which satisfies the {\em q-real} condition $J\Theta=\bar \Theta$  and $b$ is determined by $F$:
$$
b=\frac{\int_M\Omega^n\wedge \bar \Theta}{\int_M{\rm e}^F\Omega^n\wedge \bar \Theta}\,.
$$
So far there are only partial results about  the solvability of \eqref{QMAe}. In \cite{Alesker-Verbitsky (2010)} it is proved the uniqueness and an a priori $C^0$ estimate for solutions to \eqref{QMAe} when the canonical bundle of $(M,I)$ is holomorphically trivial; in \cite{Alesker-Shelukhin (2017),Sroka} the $C^0$ estimate is established
on any compact HKT manifold, without further assumptions and in \cite{GentiliVezzoni} the problem is studied on some $8$-dimensional nilmanifolds. In \cite{Alesker (2013)} it is proved that equation \eqref{QMAe} can always be solved on HKT manifolds with {\em locally flat} hypercomplex structure which admit a compatible hyperk\"ahler metric. 
Here we recall that a hypercomplex structure is called {\em locally flat} if it is locally isomorphic to $\H^n$. Manifolds of this kind were firstly considered by Sommese in \cite{S} and simply called {\em quaternionic manifolds}. Recently in \cite{DinewSroka} the quaternionic Monge-Ampère equation has been solved on hyperk\"ahler manifolds without the assumption of local flatness.

\medskip 
In the present paper we approach equation \eqref{QMAe} via the following geometric flow    
\begin{equation}\label{QPMAe}
\varphi_t=\log	\frac{(\Omega+\partial\partial_J\varphi)^n}{\Omega^n}-F\,,\quad \varphi(x,0)=0\,,
\end{equation}
where the solution $\varphi$ is supposed to satisfy $\varphi(\cdot,t)\in \mathcal H_\Omega $ for every $t$ and the subscript $t$ denotes the derivative of $\varphi$ with respect to the variable $t$. 
The same dynamic approach was pursued on K\"ahler manifolds \cite{Cao}, on Hermitian manifolds \cite{gill,sun} and on almost Hermitian manifolds \cite{Chu2}.

Our main result is the following theorem which provides an alternative proof of Alesker's Theorem \cite{Alesker (2013)}.
\begin{teor}\label{main}
Let $(M,I,J,K,g)$ be a compact HKT manifold with $(I,J,K)$ locally flat and assume that there exists a hyperk\"ahler metric $\hat g$ on $(M,I,J,K)$. Then there exists a long-time solution $\varphi\in C^{\infty}(M\times\R_+)$ to the parabolic quaternionic Monge-Amp\`ere equation \eqref{QPMAe} such that  
$$
\tilde \varphi=\varphi-\int_M\varphi\, \Omega^n\wedge \bar{\Omega}_{\hat g}^n
$$
converges in $C^{\infty}$-topology to a smooth function $\tilde\varphi_{\infty}\in C^{\infty}(M)$. Moreover if 
$$
b:=\frac{\int_M\Omega^n\wedge \bar \Omega_{\hat g}^n}{\int_M{\rm e}^F\Omega^n\wedge  \bar \Omega_{\hat g}^n}\,,
$$ 
then $(\tilde\varphi_\infty,b)$ solves the 
quaternionic Monge-Amp\`ere equation \eqref{QMAe}.
\end{teor}
 
Recently we have been made aware of the paper \cite{Zhang} where the parabolic quaternionic Monge-Amp\`ere equation is studied and its long-time behaviour is described with techniques different from ours. 

\medskip 
Now we describe the layout of the proof.
Since \eqref{QPMAe} is strongly parabolic, it admits a unique maximal solution $\varphi\in C^{\infty}(M\times [0,T))$.   
\begin{enumerate}
\item[Step 1.] From the equation we directly deduce a uniform $C^0$ bound on $\varphi_t$ (Lemma \ref{varphit}).

\vspace{0.2cm}
\item[Step 2.]  The $C^0$ estimate for solutions of the quaternionic Calabi-Yau equation \eqref{QMAe} then implies a uniform bound on ${\rm osc}\,\varphi$ (Lemma \ref{lemmatilde}).

\vspace{0.2cm}
\item[Step 3.]  We use the existence of the hyperk\"ahler metric and the local flatness of the hypercomplex structure in order to establish a uniform upper bound on $\Delta_{\hat g}\varphi$ (Lemma \ref{laplacian_estimate}). 

\vspace{0.2cm}
\item[Step 4.]  A general result in \cite{Chu} implies a uniform H\"older estimate on the second derivatives of $\varphi$, thus a classical bootstrapping argument using Schauder estimates implies $T=\infty$ and a uniform bound on $|\nabla^k\varphi|$ for $k\geq 1$ (Lemmas \ref{2nd_der} and \ref{schauder}). 

\vspace{0.2cm}
\item[Step 5.]  We prove the convergence of $\tilde \varphi$ using an argument due to Phong-Sturm \cite{PhongSturm} based on an adapted Mabuchi-type functional (Lemma \ref{phong}).  
\end{enumerate}
We point out that the local flatness of the hypercomplex structure plays a role in steps 3 and 4, while the existence of a background hyperk\"ahler metric is only used in step 3. 

\begin{rem}{\em 
Flow \eqref{QPMAe} can be regarded as a geometric flow in Hermitian Geometry. Here we assume that the canonical bundle of $(M,I)$ is trivial and we fix a $q$-real complex volume form $\Theta$ on $(M,I)$.
As shown in \cite{Alesker-Verbitsky (2010)} one has   
$$
(\Omega+\partial\partial_J\varphi)^n\wedge\bar\Theta=i^n(\omega-i\partial \bar\partial \varphi)^n\wedge \Phi\,,\quad \Omega^n\wedge\bar\Theta=i^n\omega^n\wedge \Phi 
$$
where $\omega$ is the fundamental form of $(g,I)$ and $\Phi$ is a real $(n,n)$-form which is positive in a weak sense. By setting $u=-\varphi$ we can then rewrite \eqref{QPMAe} as
\begin{equation}\label{fakeHE}
u_t=-\log	\frac{(\omega+i\partial \bar\partial u)^n\wedge\Phi}{\omega^n\wedge \Phi}+F\,,\quad u(0)=0\,. 
\end{equation}}

{\em Equation 	\eqref{fakeHE} reminds the {\em parabolic $k$-Hessian flow}
\begin{equation}\label{phongto}
u_t=\log	\frac{(\chi+i\partial \bar\partial u)^k\wedge\alpha^{n-k}}{\alpha^n}+F\,,\quad u(0)=0
\end{equation}
studied by Phong and T\^o on a complex $n$-dimensional Hermitian manifold $(M,\alpha)$ in \cite{PhongTo}, where $1\leq k\leq n$ and $\chi$ is real $k$-positive $(1,1)$-form. According to  \cite{PhongTo} \eqref{phongto} has always a long-time solution whose normalization converges in $C^\infty$-topology to a solution of the $k$-Hessian equation. 
Equation 
\eqref{fakeHE} differs from the  parabolic $n$-Hessian flow since the role of $\alpha^n$ is replaced by the form $\Phi$ which is positive in a weak sense and the theorem of Phong and T\^o cannot be directly applied.}
\end{rem}

\bigskip
\noindent {\bf Acknowledgments.} The authors are grateful to Gueo Grantcharov for useful conversations. The remark at the end of the introduction was generated by a question of Jeffrey Streets and a conversation with Misha Verbitsky and Marcin Sroka, we are very grateful to them for their interest in our paper.  

Moreover the authors would like to thank the anonymous referee who carefully read the first version of the present paper and made many very useful remarks which allowed them to considerably improve the presentation of the results.   

\section{Preliminaries}\label{preliminaries}
Let $(M,I,J,K,g)$ be a compact HKT manifold with HKT form $\Omega$. 
Let $\partial$ be the $\partial$-operator with respect to $I$ and $\partial_J:=J^{-1}\bar\partial J\colon \Lambda^{r,0}_I\to \Lambda^{r+1,0}_I$.  Then 
$$
\partial\partial_J=-\partial_J\partial\,,
$$
see \cite{Verbitsky (2002)}.
Moreover we assume  that the canonical bundle of $(M,I)$ is holomorphically trivial and we let $\Theta$ be 
a {\rm q}-real holomophic volume form on $(M, I)$. Note that, since $ \Omega $ is easily seen to be {\rm q}-real, $\Omega^n\wedge \bar \Theta$ is a real volume form; indeed, $J$ acts trivially on top forms and thus
$$
\overline{\Omega^n\wedge \bar \Theta}=J\Omega^n\wedge J\bar\Theta =\Omega^n\wedge \bar \Theta\,.
$$

 The HKT metric induces the {\em quaternionic Laplacian} operator 
$$
\Delta_g\varphi:=\frac{\partial \partial_J\varphi\wedge\Omega^{n-1}}{\Omega^n}
$$
for $\varphi\in C^{\infty}(M)$.  It is well-known that $\Delta_g$ is elliptic and it is 
straightforward to show that 
for $\eta,\psi\in C^{\infty}(M)$ we have 
$$
\int_M(\Delta_g\eta)\psi\,\Omega^n\wedge\bar \Theta  = 
\int_M\eta(\Delta_g\psi)\,\Omega^n\wedge\bar \Theta\,.
$$
Moreover the following formula will be useful: for every $\alpha, \beta \in \Lambda^{1,0}_I$
\begin{equation}\label{Lucio}
\frac{\alpha \wedge J\bar \beta \wedge \Omega^{n-1}}{\Omega^n}= - \frac{1}{2n} g(\alpha,\bar \beta)\,.
\end{equation}

\medskip 

The basic example of hyperhermitian manifold is given by an open set $A$ of $\mathbb R^{4n}$ with the standard hyperhermitian structure
$$
I_0=\begin{pmatrix}
0 & -\mathbbm{1}_n & 0 & 0\\
\mathbbm{1}_n & 0 & 0 & 0\\
0 & 0 & 0 &-\mathbbm{1}_n\\
0 & 0 & \mathbbm{1}_n & 0
\end{pmatrix}\,,\quad
J_0=\begin{pmatrix}
0 & 0 & -\mathbbm{1}_n & 0\\
0 & 0 & 0 & \mathbbm{1}_n\\
\mathbbm{1}_n & 0 & 0 & 0\\
0 & -\mathbbm{1}_n & 0 & 0
\end{pmatrix}\,,\quad
K_0=\begin{pmatrix}
0 & 0 & 0 & -\mathbbm{1}_n\\
0 & 0 & -\mathbbm{1}_n & 0\\
0 & \mathbbm{1}_n & 0 & 0\\
\mathbbm{1}_n & 0 & 0 & 0
\end{pmatrix}\,,
$$ 
where $ \mathbbm{1}_n $ is the $ n\times n $ identity matrix. In this case for an $\mathbb H$-valued function $u\colon \R^{4n}\to \mathbb H$ the following derivatives are defined     
$$
\partial_{q^r}u:=\partial_{x_0^r}u e_0-\sum_{i=1}^3 \partial_{x_i^r}u e_i \,,\quad 
\partial_{	\bar q^r}u:=\sum_{i=0}^3e_{i} \partial_{x_i^r}u \,,
$$
where to shorten the notation we denote the quaternions $1,i,j,k$ with $e_0,e_1,e_2,e_3$ and the coordinates on $\mathbb R^{4n}$ are taken as
$(x^1_{0},\dots,x^n_0,x^{1}_{1},\dots,x^n_1,x^1_2,\dots,x^n_2,x^1_3,\dots,x^{n}_{3})$ in  order to identify $\R^{4n}$ with $\mathbb H^n$. We denote by 
$$
{\rm Hyp}(n,\mathbb H)=\{U\in \mathbb H^{n,n}\,:\bar U=\,^tU\,\}
$$
the space of hyperhermitian matrices. Any $U\in {\rm Hyp}(n,\mathbb H)$ has real eigenvalues 
and we can consider the subset  ${\rm Hyp}^+(n,\mathbb H)$ of positive-definite hyperhermitian matrices.

Any hyperhermitian Riemannian metric $g$ on $(A,I_0,J_0,K_0)$ defines a smooth map $G\colon A\to {\rm Hyp}^{+}(n,\mathbb H)$,
$$
G_{rs}:=g(\partial_{q^r},\partial_{q^s})\,,
$$
where $g$ is extended $\mathbb H$-semilinearly in its components, i.e.
$$
g(X,Y\lambda)=g(X,Y)\lambda\,,\quad g(X\lambda ,Y)=\bar \lambda g(X,Y)
$$   
for every $\lambda\in\mathbb H$, $X,Y\in \Gamma(TA)$.

\medskip 
In some occasions we will make use of the following real representation of quaternionic matrices  
$\iota \colon\mathbb H^{n,n}\to \{U\in \R^{4n,4n}\,:I_0UI_0=J_0UJ_0=K_0UK_0=-U\,\}$\,,
$$
\iota(A+iB+jC+kD):=
\begin{pmatrix}
A  & B & C & D\\
-B & A & -D &C\\
-C & D & A &-B\\
-D & -C & B& A
\end{pmatrix}\,.
$$
The map $\iota$ is an isomorphism of real algebras and
$$
\iota({\rm Hyp}(n, \H)) = {\rm Hyp}(n, \R)\,,
$$
where $${\rm Hyp}(n, \R)= \{U\in \mathrm{Sym}(4n,\R)\,: I_0UI_0=J_0UJ_0=K_0UK_0=-U\,\}\,.$$

\medskip For any smooth function $u\colon A\to \mathbb R$ it is defined the {\em quaternionic Hessian matrix}  
$$
({\rm Hess}_{\mathbb H}u)_{rs}:=  u_{\bar r	s}\,,
$$
where we set $u_{\bar r s} = \partial_{\bar q^r}\partial_{ q^s} u$.
${\rm Hess}_{\mathbb H}u$ is a hyperhermitian quaternionic matrix, in particular the entries $({\rm Hess}_{\mathbb H}u)_{rr}$ are real.

\medskip  

The following lemma will be useful. We refer to \cite[Proposition 4.1]{Alesker-Verbitsky (2006)} for a proof (see also \cite{SThesis}).

\begin{lem}\label{fond}
Let $g$ be a HKT metric on $(A,I_0,J_0,K_0)$. Then the matrix associated to $g$ is 
$$
G=\kappa {\rm Hess}_{\mathbb H}u\,,
$$
where $u\in C^{\infty}(A,\R)$ is such that $\Omega=\partial\partial_J u$ is the HKT form associated to $g$ and $\kappa>0$ is a universal constant. 
\end{lem}
Next we recall the formulation of the {\em quaternionic Monge-Amp\`ere equation} \eqref{QMAe} on open sets of $\mathbb H^n$;
for a hyperhermitian $U \in \mathbb{H}^{n,n}$ we will denote by $\det U$ its {\em Moore determinant} (see \cite{Moore}).
\begin{lem}\label{formule}
Let $g$ be a HKT metric on $(A,I_0,J_0,K_0)$, $\varphi\colon A\to \R$ a smooth function and $\Omega$ the HKT form of $g$, then 
$$
(\Omega+\partial \partial_J\phi)^n=\frac{\det(G+\kappa\,{\rm Hess}_{\mathbb H}\phi)}{\det G}\,\Omega^n\,,\quad 
n\frac{\partial\partial_J\varphi\wedge \Omega^{n-1}}{\Omega^n}=
\kappa\,\Re\left( {\rm tr}(G^{-1}{\rm Hess}_{\mathbb H}\phi)\right)\,.
$$ 
\end{lem}
\begin{proof}
The first formula is \cite[Corollary 4.6]{Alesker-Verbitsky (2006)} and the second is simply obtained by linearizing the first one at the origin and using $ \det \iota (U)=(\det U)^4 $ for any hyperhermitian matrix $ U $ (see \cite[Theorem 2.4]{Alesker-Verbitsky (2006)}),
\[
\frac{d}{ds}\bigg \vert_{s=0} \log \frac{(\Omega+ \partial \partial_J (s\phi))^n}{\Omega^n}=\frac{d}{ds}\bigg \vert_{s=0}\log \frac{\det(G+\kappa \mathrm{Hess}_\H (s\phi))}{\det G}
\]
which gives
\[
\begin{split}
n\frac{\partial \partial_J \phi \wedge \Omega^{n-1}}{\Omega^n}&=\frac{d}{ds}\bigg \vert_{s=0}\log \frac{\det \iota \left( G+\kappa \mathrm{Hess}_\H (s\phi)\right)^{1/4}}{\det \iota (G)^{1/4}}=\frac{1}{4}\frac{d}{ds}\bigg \vert_{s=0}\log \frac{\det \left( \iota (G+\kappa \mathrm{Hess}_\H (s\phi))\right)}{\det \iota  (G)}\\
&= \frac{\kappa}{4} \mathrm{tr}\left( \iota (G)^{-1}\iota (\mathrm{Hess}_\H\phi)  \right)= \kappa \Re \left(\mathrm{tr}\left( G^{-1} \mathrm{Hess}_\H\phi  \right)\right)
\end{split}
\]
as claimed.
\end{proof}

Finally, we provide a lemma which will be helpful in the proof of the main theorem.

\begin{lem}\label{Deltalogdet}
Let $U\colon A\to {\rm Hyp}^+(n,\mathbb H)$ be a smooth map   and assume that there exists $p\in A$ such that $U(p)$ is diagonal. Let $\hat g$ be a hyperhermitian metric on $A$ such that the induced matrix $\hat G$ is the identity. 
Then  
$$
\Delta_{\hat g}\log(\det U)=-\frac{\kappa}{n}\sum_{r,s,t=1}^n\sum_{i=0}^3 \frac{1}{U_{ss}} \frac{1}{U_{tt}} |U_{st,x^r_i} |^2  +\sum_{s=1}^n\frac{1}{U_{ss}} \Delta_{\hat g} U_{ss}
$$ 
at $p$, where the subindex \lq\lq $x^r_{i}$\rq\rq denotes the derivative with respect to the corresponding real coordinate. 
\end{lem}
\begin{proof}
Since $ \det \iota (U)=(\det U)^4 $ we directly compute
\[
\begin{aligned}
\partial_{\bar q^r}\partial_{q^r} \log(\det U)&=\sum_{i=0}^3\partial_{x^r_i}^2 \log(\det U)=\frac{1}{4}\sum_{i=0}^3\partial_{x^r_i}^2 \log(\det \iota(U))=\frac{1}{4}\sum_{i=0}^3\partial_{x^r_i} \mathrm{tr}\left( \iota(U)^{-1} \iota(U)_{,x^r_i} \right)\\
&=\frac{1}{4}\sum_{i=0}^3 \mathrm{tr}\left( -\iota(U)^{-1}\iota(U)_{,x^r_i}\iota(U)^{-1} \iota(U)_{,x^r_i}+ \iota(U)^{-1}\iota(U)_{,x^r_ix^r_i} \right)\\
&=\frac{1}{4}\sum_{i=0}^3 \mathrm{tr}\left(\iota \left( -U^{-1}U_{,x^r_i}U^{-1} U_{,x^r_i}+ U^{-1}U_{,x^r_ix^r_i}\right) \right)\\
&=\sum_{i=0}^3 \Re \left(\mathrm{tr}\left( -U^{-1}U_{,x^r_i}U^{-1} U_{,x^r_i}+ U^{-1}U_{,x^r_ix^r_i} \right)\right)
\end{aligned}
\]
and at the point $p$ where $U$  takes a diagonal form
\[
\begin{aligned}
\Delta_{\hat g} \log(\det U)&=\frac{\kappa}{n}\sum_{r,s,t=1}^n \sum_{i=0}^3 \Re \left(-U^{ss}U_{st,x^r_i}U^{tt} U_{ts,x^r_i}+ U^{ss}U_{ss,x^r_ix^r_i} \right)\\
&=-\frac{\kappa}{n}\sum_{r,s,t=1}^n \sum_{i=0}^3 \frac{1}{U_{ss}}\frac{1}{U_{tt}}|U_{st,x^r_i}|^2+ \sum_{s=1}^n\frac{1}{U_{ss}}\Delta_{\hat g}U_{ss}
\end{aligned}
\]
and the claim follows.
\end{proof}

\section{Proof of the Result}
Let $(M,I,J,K,g)$ be a HKT manifold with HKT form $\Omega$. Every $\varphi \in \mathcal{H}_\Omega$ induces a HKT metric $g_\varphi$ and a quaternionic Laplacian $\Delta_\varphi:=\Delta_{g_\varphi}$. 
Consider the operator
$$
P\colon \mathcal{H}_\Omega \to C^{\infty}(M)\,,\quad  P(\varphi)=\log	\frac{(\Omega+\partial\partial_J\varphi)^n}{\Omega^n}-F \,.
$$
The first variation of $P$ is  
$$
P_{*|\varphi}(\psi)=n\frac{\partial\partial_J\psi\wedge (\Omega+\partial\partial_J\varphi)^{n-1}}{(\Omega+\partial\partial_J\varphi)^{n}}=n\Delta_\varphi\psi\,.
$$
Since $\Delta_{\varphi}$ is a strongly elliptic operator, equation \eqref{QPMAe} is always well-posed and it admits a unique maximal solution $\varphi\in C^{\infty}(M\times\mathbb[0,T))$. Assume further that the canonical bundle of $(M,I)$ is holomorphically trivial and let $\Theta\in\Lambda^{2n,0}_I$ be a q-real holomorphic volume form. 
We then denote
$$
\tilde\varphi= \varphi-\int_M\varphi\,\Omega^n\wedge \bar\Theta\,. 
$$
We start by proving $ C^0 $ bounds for the time derivatives $ \phi_t $ and $ \tilde \phi_t $ and then use these to prove the $ C^0 $ estimate for $ \tilde \phi $. In what follows we denote by $C$ all the uniform constants (which may be different from line to line). 

\begin{lem}\label{varphit}
There exists a uniform constant $C$ such that 
$$
|\varphi_t(x,t)|\leq C\,,\quad |\tilde \varphi_t(x,t)|\leq C\,
$$
for every $(x,t)\in M\times [0,T)$. 
\end{lem}
\begin{proof}
Since 
$$
\frac{\partial}{\partial t} \log	\frac{(\Omega+\partial\partial_J\varphi)^n}{\Omega^n}=n
\frac{\partial\partial_J\varphi_t\wedge \Omega^{n-1}_\varphi}{\Omega^{n}_\phi}=n\Delta_\varphi\varphi_t\,,
$$
we have 
$$
\varphi_{tt}=n\Delta_\varphi\varphi_t
$$
and the parabolic maximum principle implies the a priori $C^{0}$ estimate for $\varphi_t$.  The estimate on $\tilde\varphi_t$ immediately follows.  
\end{proof}

\begin{lem}\label{lemmatilde}
We have 
$$
\max_M \varphi-\min_M \varphi\leq C
$$
and 
$$
|\tilde \varphi|\leq C\,,
$$
for a uniform constant $C$.
\end{lem}
\begin{proof}
Since $ \abs{\phi_t} $ is bounded and
$$
(\Omega+\partial\partial_J\varphi)^n={\rm e}^{F+\varphi_t}\,\Omega^n\,,
$$
for every fixed $t$, $\varphi(\cdot,t)$ solves the quaternionic Monge-Amp\`ere equation \eqref{QMAe} with datum $F+\varphi_t$. In view of the $C^0$ estimate for solutions to the quaternionic Monge-Amp\`ere equation 
\cite{Alesker-Shelukhin (2017),Alesker-Verbitsky (2010),Sroka}, $\varphi$ satisfies the bound
\begin{equation}
\label{osc}
\max_M \varphi-\min_M \varphi\leq C\,,
\end{equation}
where $C$ depends only on  $(M,I,J,K,g)$ and an upper bound of $\max |F+\varphi_t|$. Therefore Lemma \ref{varphit} implies that the constant $C$ in \eqref{osc} may be chosen so that it only depends on  $(M,I,J,K,g)$ and an upper bound of $\max |F|. $
Now, let $ (x,t)\in M\times [0,T)$, since $ \tilde \phi $ is normalized, there exist $(y,t)$ such that $ \tilde \phi(y,t)=0 $, and thus we obtain $ |\tilde \phi(x,t)|=|\tilde \phi(x,t)-\tilde \phi(y,t)|=|\phi(x,t)-\phi(y,t)|\leq C $ and the claim follows. 
\end{proof}


\begin{lem}\label{laplacian_estimate}
Assume that $(I,J,K)$ is locally flat and that there exists a hyperk\"ahler  metric $\hat g$ on $(M,I,J,K)$. Then
$$
\Delta_{\hat g}\varphi\leq C\,,
$$
for a uniform constant $C$. 
\end{lem}

\begin{proof}
Let 
$$
Q=2\sqrt{\frac{1}{n}\mathrm{tr}_{ \hat g}g_{\varphi}}-\varphi\,.
$$
Fix $T'<T$ 
and let $(x_0,t_0)$ be a point where $Q$ achieves its maximum in $M\times [0,T']$. We may assume without loss of generality that $t_0>0$. Since $(I,J,K)$ is locally flat, then  in a neighborhood of $x_0$ we can locally identify $M$ with an open set $A$  of $\mathbb H^n$. Let $G$ and $\hat G$ be the hyperhermitian matrices in $A$ induced by $g$ and $\hat g$ respectively. We may further assume that $G={\rm Hess}_{\mathbb H}v$ in $A$, that $\hat G$ is the identity in $A$ and that $U={\rm Hess}_{\mathbb H}(v+\kappa\varphi)$ is diagonal at $x_0$. Let $u=v+\kappa\varphi$. Then in $A$ we have 
$$ 
Q=2\sqrt{\Delta_{\hat g} u}-\varphi\,.
$$
Computing at $(x_0,t_0)$, we have 
$$
\Delta_{\varphi} Q=\frac{\kappa}{n}
\frac{1}{\sqrt{\Delta_{\hat g} u}} \sum_{r=1}^n \frac{1}{u_{r\bar r}}\left(-\frac12 \frac{1}{\Delta_{\hat g} u}|\Delta_{\hat g}\,u_r|^2+\Delta_{\hat g} u_{r	\bar r}\right)-\Delta_{\varphi}\varphi
$$
and, applying Lemma \ref{formule} and Lemma \ref{Deltalogdet} we infer
$$
\begin{aligned}
\partial_tQ=&\,\frac{1}{\sqrt{\Delta_{\hat g} u}}\Delta_{\hat g}  \varphi_t-\varphi_t=
\frac{1}{\sqrt{\Delta_{\hat g} u}}\Delta_{\hat g}  (\log \det (U)-\log \det (G)-F)-\varphi_t\\
=&\,
\frac{1}{\sqrt{\Delta_{\hat g} u}}
\left(-\frac{\kappa}{n}\sum_{r,s,t=1}^n\sum_{i=0}^3 \frac{1}{u_{s \bar s}} \frac{1}{u_{t \bar t}} |u_{s \bar t,x^r_i} |^2  +\sum_{r=1}^n\frac{1}{u_{r \bar r}} \Delta_{\hat g} u_{r \bar r}-\Delta_{\hat g}\log \det (G) -\Delta_{\hat g} F\right)-
\varphi_t
\end{aligned}
$$
which implies 
\begin{multline*}
\partial_tQ-\frac{n}{\kappa}\Delta_{\varphi} Q=\\
\frac{1}{\sqrt{\Delta_{\hat g} u}}
\left(\frac{1}{2\Delta_{\hat g} u} \sum_{r=1}^n\frac{1}{u_{r\bar r}}|\Delta_{\hat g}\,u_r|^2-\frac{\kappa}{n}\sum_{r,s,t=1}^n\sum_{i=0}^3 \frac{1}{u_{s \bar s}} \frac{1}{u_{t \bar t}} |u_{s \bar t,x^r_i} |^2 -\Delta_{\hat g} (F+\log\det(G))\right)+\frac{n}{\kappa}\Delta_{\varphi}\varphi
-\varphi_t\,.
\end{multline*}
Using the Cauchy-Schwarz inequality and \cite[Proposition 3.1]{Alesker (2013)} we obtain
\[
\begin{aligned}
\sum_{r=1}^n\frac{1}{u_{r\bar r}}|\Delta_{\hat g} u_{r}|^2 & =\sum_{r=1}^n\sum_{i=0}^3\frac{1}{u_{r\bar r}}(\Delta_{\hat g} u_{x^r_i})^2=\frac{\kappa^2}{n^2}\sum_{r=1}^n\sum_{i=0}^3\frac{1}{u_{r\bar r}}\left(\sum_{s=1}^n\frac{\sqrt{u_{s\bar{s}}}}{\sqrt{u_{s\bar{s}}}}u_{s\bar{s},x^r_i}\right)^2\\
&\leq\frac{\kappa^2}{n^2}\sum_{r,s,t=1}^nu_{t\bar{t}}\sum_{i=0}^3\frac{1}{u_{r\bar r}} \frac{1}{u_{s\bar{s}}}(u_{s\bar{s},x^r_i})^2=\frac{\kappa}{n}\Delta_{\hat{g}} u\sum_{r,s=1}^n\sum_{i=0}^3\frac{1}{u_{r\bar r}} \frac{1}{u_{s\bar{s}}}(u_{s\bar{s},x^r_i})^2\\
&\leq 2\frac{\kappa}{n} \Delta_{\hat g} u \sum_{r,s,t=1}^n \sum_{i=0}^3 \frac{1}{u_{ s\bar s}} \frac{1}{u_{ t\bar t}} | u_{s\bar t, x^r_i} |^2
\end{aligned}
\]
i.e.
$$
\frac{1}{2\Delta_{\hat g} u}\sum_{r=1}^n\frac{1}{u_{r\bar r}}|\Delta_{\hat g} u_{r}|^2\leq \frac{\kappa}{n}\sum_{r,s,t=1}^n\sum_{i=0}^3 \frac{1}{u_{s \bar s}} \frac{1}{u_{t \bar t}} |u_{ s \bar t,x^r_i} |^2 \,,
$$
from which it follows
$$
0\leq \partial_tQ-\frac{n}{\kappa}\Delta_{\varphi} Q\leq \frac{n}{\kappa}\Delta_{\varphi}\varphi-\frac{\Delta_{\hat g} (F+\log\det(G))}{\sqrt{\Delta_{\hat g} u}}-\varphi_t\leq 
\frac{n}{\kappa}-\frac{n}{\kappa^2}\Delta_{\varphi}v-\frac{\Delta_{\hat g} (F+\log\det(G))}{\sqrt{\Delta_{\hat g} u}}-\varphi_t
$$
at $(x_0,t_0)$, 
where we have used that it is a maximum point as well as the relation 
$$
\Delta_{\varphi}\varphi=1-\frac{1}{\kappa}\Delta_{\varphi}v\,. 
$$
Hence
$$
 \frac{n}{\kappa^2}\Delta_{\varphi}v \leq \frac{n}{\kappa}- \frac{\Delta_{\hat g} (F+\log\det(G))}{\sqrt{\Delta_{\hat g} u}}-\varphi_t
$$
at $(x_0,t_0)$. Since $|\varphi_t|$ is uniformly bounded we obtain 
\begin{equation}\label{boh2} 
 \Delta_{\varphi}v(x_0,t_0) \leq C + \frac{C}{\sqrt{\sum_{r=1}^n u_{r\bar r}(x_0,t_0)}}
\end{equation}
for a uniform constant $C$. 
In terms of $u$ and $G$ equation \eqref{QPMAe} writes as
$$
\frac{1}{\kappa}u_t=\log\det\left(U \right)-\log\det (G)-F
$$
and then 
$$
\frac{1}{\kappa}u_t(x_0,t_0)=\log \prod_{r=1}^{n} u_{r\bar r}(x_0,t_0)-\log \det (G(x_0))-F(x_0)\,.
$$
Lemma \ref{varphit} implies that $|u_t|$ is uniformly bounded and we deduce that 
$$\frac{1}{C} \leq \prod_{r=1}^{n}u_{r\bar r}(x_0,t_0) \leq C\,.$$ 
Thus in particular by the geometric-arithmetic mean inequality we have  $\sum_{r=1}^{n}u_{r\bar r}(x_0,t_0) \geq C$.\\
Since 
$$
 \Delta_{\varphi}v(x_0,t_0)=\frac{\kappa}{n}\sum_{r=1}^n\frac{1}{u_{r\bar r}(x_0,t_0)}v_{r\bar r }(x_0)\,,
$$ 
by \eqref{boh2} we finally deduce 
$$
\sum_{r=1}^n\frac{1}{u_{r\bar r}(x_0,t_0)}\leq C\,.
$$

Therefore
$$
\Delta_{\hat g}u(x_0,t_0)=\frac{\kappa}{n}\sum_{r=1}^n u_{r\bar r}(x_0,t_0)\leq \frac{\kappa}{n} \frac{1}{(n-1)!}
\left(
\sum_{r=1}^n\frac{1}{u_{r\bar r}(x_0,t_0)}
\right)^{n-1}\prod_{r=1}^{n}u_{r\bar r}(x_0,t_0)\leq C\,.
$$
It follows  
$$
2\sqrt{\Delta_{\hat g} u(x,t)}\leq C+\varphi(x,t)-\varphi(x_0,t_0)\leq C+\mathrm{osc}\,\varphi\qquad \mbox{ in } M\times [0,T']\,,
$$
from which, using Lemma \ref{lemmatilde}, we get 
$$
\Delta_{\hat g} u\leq C
$$
for a uniform $C$ and the claim is proved.  
\end{proof}

\begin{lem}
\label{2nd_der}
Assume that  $(I,J,K)$ is locally flat and that there exists a hyperhermitian metric $\hat g$ on $(M,I,J,K)$ such that 
$$
\Delta_{\hat g}\varphi\leq C
$$
for a uniform constant $C$. Then for $0<\alpha<1$  we have 
$$
\|\nabla^2\varphi\|_{C^{\alpha}}\leq C 
$$  
for a uniform constant $C$. 
\end{lem}
\begin{proof}
We prove the result by applying \cite[Theorem 5.1]{Chu}.  
First we state some algebraic preliminaries (which are analogous to the complex case \cite[Section 2]{Tosatti et al.}). Note that, in the notation introduced in section \ref{preliminaries}, the real representation $\iota \colon\mathbb H^{n,n}\to \{H\in \R^{4n,4n}\,:I_0HI_0=J_0HJ_0=K_0HK_0=-H\,\}$ of quaternionic matrices is monotonic in the sense that when $H_1,H_2$ are hyperhermitian one has
$$
H_1\leq H_2\Rightarrow \iota(H_1)\leq \iota(H_2)\,,
$$
where $ H_1 \leq H_2 $ means that all the eigenvalues of $H_2-H_1$ are non-negative.

Let ${\rm p}\colon \R^{4n,4n}\to \{H\in \R^{4n,4n}\,:I_0HI_0=J_0HJ_0=K_0HK_0=-H\,\}$ be the projection defined as 
$$
{\rm p}(N):=\frac14(N-I_0NI_0-J_0NJ_0-K_0NK_0)\,.
$$
Then for any real valued smooth function $f$ and any hyperhermitian matrix $ H $ we have
$$
\iota({\rm Hess}_\mathbb{H}f)=4 {\rm p}({\rm Hess}_\mathbb{R}f)\,,\quad \det(\iota (H))=(\det H)^4\,.
$$

Thus, once local quaternionic coordinates are fixed, working as in the proof of Lemma \ref{laplacian_estimate}, we can rewrite equation \eqref{QPMAe} as 
$$
\frac{1}{\kappa}u_t=\frac14 \log \det\left(4{\rm p}({\rm Hess}_\mathbb{R}u)\right)-\log \det(G)-F\,,
$$
where  $u=\kappa v+\kappa\varphi$ and $v$ is a HKT potential of $\Omega$.
We rewrite the last equation as 
\begin{equation}\label{eqnchun-li}
\frac{1}{\kappa}u_t=\Phi({\rm p}({\rm Hess}_\mathbb{R}u))-\log \det(G)-F
\end{equation}
where  for  $N\in{\rm Sym}(4n,\R)$ such that $\det N>0$ we set 
$$
\Phi(N)=\frac14 \log \det \left(4 N\right)\,.
$$
Fix positive constants $C_1<C_2$ and let 
%
%
%
%
$$
\mathcal{E}:=\left\{N\in {\rm Sym}(4n,\R)\,:\,C_1 {\mathbbm 1}_{4n}\leq N\leq C_2{\mathbbm 1}_{4n}\,\right\}\,.
$$
Then $\mathcal{E}$ is a compact convex subset of  ${\rm Sym}(4n,\R)$. 
We observe that $ \Phi $ and ${\rm p}$ satisfy the assumptions in \cite[Theorem 5.1]{Chu}.  Indeed
\begin{itemize}
\item $\Phi$ is uniformly elliptic in $\mathcal E$;

\vspace{0.1cm}
\item $\Phi$ is concave in $\mathcal E$;

\vspace{0.1cm}
\item ${\rm p}$ is linear;

\vspace{0.1cm}
\item if $N\geq 0$, then ${\rm p}(N)\geq 0$ and $C^{-1} \|N\|\leq \|{\rm p}(N)\|\leq C\|N\|$ for $C$ uniform.  
\end{itemize} 
Therefore, if we show that ${\rm p}({\rm Hess}_{\R}u)\in \mathcal{E}$ for a suitable choice of $C_1$ and $C_2$ equation \eqref{eqnchun-li} belongs to the class of equations considered in \cite[Theorem 5.1]{Chu}. 

\medskip
Without loss of generality we can fix $x_0\in M$ and assume that $\hat G$ is the identity at $ x_0 $. Our assumption $ \Delta_{\hat g} \phi \leq C $ implies
\begin{equation}\label{bound}
\sum_{r=1}^n u_{r\bar r}\leq C
\end{equation}
at $ x_0 $ for a uniform $ C>0 $ and thus
$$
{\rm Hess}_{\mathbb H}u\leq C {\mathbbm 1}_n\,.
$$
On the other hand,  equation \eqref{QPMAe} writes as
$$ 
\frac{1}{\kappa}u_t=\log \det({\rm Hess}_{\mathbb H}u)-\log \det(G)-F\,.
$$
Thus by Lemma \ref{varphit}
$$ 
\prod_{i=1}^n \lambda_i=\det({\rm Hess}_{\mathbb H}u)\geq \det(G)\mathrm{e}^{-\frac{1}{\kappa} |u_t|+F} \geq C\,,
$$
where $ \lambda_1,\dots,\lambda_n $ are the eigenvalues of $ {\rm Hess}_{\mathbb H}u$ and $ C>0 $ is a uniform constant. From \eqref{bound} we also infer $ \sum_{i=1}^n \lambda_i \leq C $ at $ x_0 $ which then implies a uniform lower bound for each $ \lambda_i $ at the point $ x_0 $, but such bound does not depend on $ x_0 $.

Therefore 
$$
C_1{\mathbbm 1}_n\leq {\rm Hess}_{\mathbb H}u\leq C_2{\mathbbm 1}_n\,.
$$
By applying $\iota$ we get  
$$
C_1{\mathbbm1 }_{4n}\leq 4{\rm p}({\rm Hess}_{\mathbb R}u)\leq C_2{\mathbbm 1}_{4n}\,.
$$
Then we can work as in the proof of \cite[Lemma 6.1]{Chu2}.

We assume that the domain of $u$ is $B	\times [0,T)$ with $B$ diffeomorphic to the unit ball in $\R^{4n}$. If $T<1$, then Lemma \ref{varphit} implies 
$$
|u|\leq CT+C\leq C
$$
for a uniform $C$ and \cite[Theorem 5.1]{Chu} implies the result. If $T\geq 1$ we define, for any $ a\in (0,T-1) $
$$
\hat u(x,t): =u(x,t+a)-\inf_{B\times [a,a+1)} u(x,t)
$$
for all $ t\in [0,1) $. We immediately deduce
$$
\frac{1}{\kappa}\hat u_t=\log \det ({\rm Hess}_\H \hat u)- \log \det (G) -F\,, \qquad   \sup_{B\times [0,1)} | \hat u(x,t)| \leq C\,.
$$
Invoking again \cite[Theorem 5.1]{Chu}, chosen $ \epsilon\in (0,\frac{1}{2}) $ and $ \alpha \in (0,1) $ we have
$$
\|\nabla^2u\|_{C^{\alpha}(B\times[a+\epsilon,a+1))}=
\|\nabla^2\hat u\|_{C^{\alpha}(B\times [\epsilon,1))}\leq C
$$
where the constant $ C $ depends on $ \epsilon $ and $ \alpha $. As $ a $ was chosen arbitrarily in $ (0,T-1) $ we have
$$
\|\nabla^2 u\|_{C^{\alpha}(B\times[\epsilon,T))}\leq C\,,
$$
and the lemma follows.
\end{proof}

\begin{lem}\label{schauder}
Assume that there exists $0<\alpha<1$ such that
$$
\|\nabla^2\varphi\|_{C^{\alpha}}\leq C 
$$  
for a uniform constant $C$. 
Then $T=\infty$ and for every $k\geq 1$
$$
\|\nabla ^k\varphi\|_{C^{0}}\leq C 
$$  
for a uniform constant $C$. 
\end{lem}
\begin{proof}
Our assumptions imply that the spatial derivatives of $\varphi$ satisfy a uniformly parabolic equation and uniform bounds on $ \|\nabla ^k\varphi\|_{C^{0}} $ with $k\geq 1 $ follow by Schauder theory and a standard bootstrapping argument.

Now we shall prove the long-time existence. 
Assume by contradiction that the maximal time interval $[0,T)$ of existence of $\varphi$ is bounded. Then the achieved estimates and short-time existence would allow us to extend $ \phi $ past $ T $, which is a contradiction, thus $ T=\infty $.
\end{proof}

\begin{lem}\label{phong}
Assume $T=\infty$ and that $\|\nabla^k\varphi\|_{C^0}$ is uniformly bounded for every $k\geq 1$. Then 
$$
\tilde \varphi:=\varphi-\int_M \varphi \, \Omega^n \wedge \bar \Theta
$$
converges in $C^{\infty}$-topology to a smooth function $\tilde \varphi_\infty$. Moreover if  
$$
b:=\frac{\int_M\Omega^n\wedge \bar \Theta}{\int_M{\rm e}^F\Omega^n\wedge  \bar \Theta}\,,
$$ 
then $(\tilde\varphi_\infty,b)$ solves the 
quaternionic Monge-Amp\`ere equation \eqref{QMAe}.

\end{lem}

\begin{proof}
Let 
$$
f(t):=\int_M \varphi_t \,\Omega_{\varphi}^n\wedge \bar \Theta = \int_M\log \frac{\Omega_\varphi^n}{\Omega^n}\,\Omega_{\varphi}^n\wedge \bar \Theta-\int_{M} F\,\Omega_{\varphi}^n\wedge\bar\Theta\,.
$$
Using \eqref{Lucio} we have 
$$
\begin{aligned}
f'&=n\int_M \left(\Delta_{\varphi}\varphi_t+\log \frac{\Omega_\varphi^n}{\Omega^n}\Delta_{\varphi}\varphi_t-F\Delta_{\varphi}\varphi_t\right) \,\Omega_{\varphi}^n\wedge\bar\Theta
\\
&=n\int_M \varphi_t \Delta_{\varphi}\varphi_t\,\Omega_{\varphi}^n\wedge\bar\Theta
=n\int_M \varphi_t \partial\partial_J\varphi_t \wedge \Omega_{\varphi}^{n-1}\wedge\bar\Theta
=-n\int_M \partial \varphi_t \wedge \partial_J\varphi_t\,\wedge \Omega_{\varphi}^{n-1}\wedge\bar\Theta \\
& = - \frac{1}{2} \int_M |\partial\varphi_t|^2_{g_{\varphi}} \Omega_\phi^n\wedge \bar \Theta\,.
\end{aligned}
$$
Differentiating again we obtain 
$$
f''=- \frac12 \int_M \tfrac{\partial}{\partial t}|\partial\varphi_t|^2_{g_{\varphi}} \Omega_\phi^n\wedge \bar \Theta
-\frac{n}{2} \int_M |\partial\varphi_t|^2_{g_{\varphi}}\Delta_{\varphi}\varphi_t\, \Omega_\phi^n\wedge \bar \Theta\,.
$$
Now 
\begin{equation}
\label{der_norm}
\frac{\partial}{\partial t} |\partial\varphi_t|^2_{g_{\varphi}}=
-g_{\varphi}\left(\tfrac{\partial}{\partial t}g_{\varphi}, \partial \varphi_t\otimes\bar\partial\varphi_t\right)+2\Re\,g_{\varphi}(\partial\varphi_{tt},\bar \partial\varphi_{t})\,.
\end{equation}
For the first term of \eqref{der_norm} Cauchy-Schwarz inequality gives
$$
-g_{\varphi}\left(\tfrac{\partial}{\partial t}g_{\varphi}, \partial \varphi_t\otimes\bar\partial\varphi_t\right)
\leq |\tfrac{\partial}{\partial t}g_\varphi|_{g_\varphi}\,|\partial \varphi_t|_{g_\varphi}^2 \leq C\,|\partial \varphi_t|_{g_\varphi}^2
$$
because $\Omega_\varphi$ and $g_{\varphi}$ are related by \eqref{Omega_g} and $\Omega_{\varphi}$ and $\tfrac{\partial}{\partial t}\Omega_{\varphi}$ are uniformly bounded in $C^{k}$-norm for every $k$.\\
For the second term of \eqref{der_norm} using \eqref{Lucio} again we have
$$
\begin{aligned}
-n\Re\int_M g_{\varphi}(\partial\Delta_\varphi\varphi_t,\bar \partial\varphi_{t})  \Omega_\phi^n\wedge \bar \Theta
=- 2n^2 \Re\,\int_M \partial\Delta_\varphi\varphi_t \wedge \partial_J \varphi_{t} \wedge \Omega_\phi^{n-1}\wedge \bar \Theta\\
= 2n^2 \Re\,\int_M \Delta_\varphi\varphi_t \, \partial\partial_J \varphi_{t} \wedge \Omega_\phi^{n-1}\wedge \bar \Theta
= 2n^2 \int_M (\Delta_\varphi\varphi_t)^2\,  \Omega_\phi^{n}\wedge \bar \Theta
\end{aligned}
$$
therefore 
$$
f'' \geq -C\int_{M}|\partial \varphi_t|_{g_\varphi}^2\,\Omega_\phi^n\wedge \bar \Theta\,.
$$

Thus we have a non increasing smooth function $f: [0,+\infty) \to \mathbb \R$ which is bounded from below and such that $f''(t) \geq C f'(t)$ for some positive constant $C$.
This implies that $\lim_{t \to +\infty} f'(t) = 0$, i.e.
\begin{equation}\label{convergence}
\lim_{t\to \infty} \int_M|\partial \varphi_t|^2_{g_\phi}\Omega_\varphi^n\wedge\bar\Theta=0\,.
\end{equation}

Now, $\tilde \varphi$ has a uniform $C^\infty$ bound and Ascoli-Arzel\`a theorem implies that there exists a sequence 
$\{t_k\}\subseteq \R$, $t_k\to \infty$ such that $\tilde \varphi(\cdot, t_k)$ converges to some $\tilde \varphi_\infty$ in $C^\infty$-topology. 
Since 
$$
\tilde \varphi_{t}=\log\frac{\Omega_{\tilde \varphi}^n}{\Omega^n}-F-\int_M \left(\log\frac{\Omega_{\tilde \varphi}^n}{\Omega^n}-F\right)\Omega^n\wedge \bar \Theta\,,
$$
by \eqref{convergence} we get
$$ 
0= \lim_{t\to \infty} \int_M |\partial \tilde \varphi_t|^2_{g_{\tilde \varphi}}\Omega_{\tilde\varphi}^n\wedge\bar\Theta = \int_M \left|\partial\left(\log\frac{\Omega_{\tilde \varphi_\infty}^n}{\Omega^n}-F\right)\right|^2_{g_{\tilde \varphi}}\Omega_{\tilde \varphi_\infty}^n
\wedge \bar\Theta \,.
$$
It follows that
$$
\log\frac{\Omega_{\tilde \varphi_\infty}^n}{\Omega^n}-F=C
$$ 
for some constant $C$, so that 
$$
\Omega_{\tilde \varphi_\infty}^n={\rm e}^{F+C}\Omega^n\,.
$$
This means that $(\tilde \varphi_\infty,{\rm e}^C)$ solves the quaternionic Calabi-Yau equation. Finally, we prove that $\lim_{t\to \infty}\tilde\varphi=\tilde\varphi_\infty$. Assume by contradiction that there exists $\epsilon>0$ and a sequence $t_k\to \infty$ such  that 
$$
\|\tilde \varphi(\cdot, t_k)-\tilde \varphi_{\infty}\|_{C^{\infty}}\geq \epsilon
$$  
for every $t_k$. We may assume that $\tilde \varphi(\cdot, t_k)$ converges in $C^{\infty}$-topology to $\tilde \varphi_{\infty}'$. 
Hence 
$$
\Omega_{\tilde \varphi'_\infty}^n={\rm e}^{F+C'}\Omega^n\,.
$$
Since $\tilde \varphi_\infty$ and $\tilde \varphi'_\infty$ solve the same quaternionic Calabi-Yau equation, from uniqueness follows $\tilde \varphi_\infty=\tilde \varphi'_\infty$ and the lemma is proved. 
\end{proof}

\begin{proof}[Proof of Theorem $\ref{main}$]
We put together Lemmas \ref{varphit}--\ref{phong} proved in this section. Lemmas \ref{varphit},\ref{lemmatilde},\ref{laplacian_estimate} imply that if $\varphi$ solves \eqref{QPMAe}, its quaternionic Laplacian $\Delta_{\hat g}\varphi$ with respect to the background hyperk\"ahler metric $\hat g$ has a uniform upper bound. Hence Lemmas \ref{2nd_der} and \ref{schauder} can be applied and \eqref{QPMAe} has a long-time solution $\varphi$ such that $\|\nabla^k\varphi\|_{C^0}$ is bounded for every $k\geq 1$. Therefore, taking $\Theta=\Omega_{\hat g}^n$, Lemma \ref{phong} implies the last part of the statement. 
\end{proof}

\section{Further developments}

On a hypercomplex manifold $(M,I,J,K)$ with a HKT form $\Omega$ a $(1,0)$-form $\theta$ is defined by the relation
$$
\bar\partial \Omega^n=\bar\theta\wedge\Omega^n\,. 
$$
If the canonical bundle is holomorphically trivial then we can take $h\in C^{\infty}(M)$ such that $\partial J\bar \theta=\partial \partial_J h\,.$\\
Now the proof of Lemma \ref{phong} suggests to consider on a ${\rm SL}(n,\H)$-manifold with holomorphic q-real volume form $\Theta$ the operator $\mathcal M$ acting on HKT forms in the
$\partial\partial_J$-class of  a fixed HKT form $\Omega_0$ as 
$$
\mathcal M(\Omega_\varphi):=
\int_M\log \frac{\Omega_\varphi^n}{\Omega^n_0}\,\Omega_{\varphi}^n\wedge \bar \Theta-\int_{M} h\,\Omega_{\varphi}^n\wedge\bar\Theta\,.
$$ 
This is related to the following geometric flow of HKT forms 
\begin{equation}\label{QRF}
\partial_t\Omega=-\partial J\bar \theta\,,\quad \Omega(0)=\Omega_0\,.
\end{equation}
Indeed working as in the proof of Lemma \ref{phong} one can observe that $\mathcal M$ is decreasing along flow \eqref{QRF}, thus $\mathcal M$ plays a role similar to that of the Mabuchi functional in Calabi-Yau geometry. It is easy to prove that the gradient flow of $\mathcal M$ can  be expressed in terms of the quaternionic potential $\varphi$ as  
\begin{equation}\label{CF}
\varphi_t=\frac{J \bar \partial_J\bar\theta_\varphi\wedge \Omega_\varphi^{n-1}}{\Omega^n_\varphi}\,,\quad \varphi(0)=0
\end{equation}
and the fixed points of $\mathcal M$ are the HKT forms in the $\partial\partial_J$-class of $\Omega_0$ which are balanced, i.e. which induce a balanced metric. 

From this perspective we believe that the operator $\mathcal M$ and the flow \eqref{CF} could give new insights in the search of canonical HKT metrics and this will be the subject of a future work.

\end{document}